\documentclass[12pt, reqno]{amsart}
\usepackage{amsmath, amsthm, amscd, amsfonts, amssymb, graphicx, color}

\newtheorem{df}{Definition}[section]
\newtheorem{thm}[df]{Theorem}

\begin{document}
\setcounter{page}{1}

\title[Joint Spectra]{Joint spectra of representations \\
of Lie algebras by compact operators}
\author {Enrico Boasso}



\begin{abstract}Given $X$ a complex Banach space, $L$ a complex nilpotent 
finite dimensional Lie algebra, and
$\rho\colon L\to L(X)$, a representation of $L$ in $X$ such that 
$\rho (l)\in K(X)$ for all $l\in L$, the Taylor, the S\l odkowski,
the Fredholm, the split and the Fredholm split joint spectra of the representation
$\rho$ are computed. \par
\end{abstract}

\maketitle
\section{Introduction}

\indent Many of the well known joint spectra defined for commuting
tuples of operators were extended to solvable Lie algebras of operators, or more generally,
to representations of solvable Lie algebras in Banach spaces. For example, among such
joint spectra, 
it is important to recall the Taylor, the S\l odkowski, the Fredholm, the split and
the Fredholm split joint spectra; see [14], [13], [8], [10], [7], [9],
[5], [2], [11], [12] and [4]. \par

\indent In the work [3]  nilpotent Lie algebras of linear 
transformations defined in finite dimensional Banach spaces were considered and the
S\l odkowski joint spectra, in particular the Taylor joint spectrum, of such an algebra were computed. In 
[1] this characterization was extended to representations of nilpotent Lie 
algebras by compact operators defined in Banach
spaces, though only for the Taylor joint spectrum.\par

\indent In this article representations of nilpotent Lie algebras by compact 
operators defined
in a finite or infinite dimensional Banach space are considered and all the above mentioned joint spectra of such a representation are computed. These results extend the 
characterizations of [3] and [1].\par

\indent The paper is organized as follows. In section 2 some
definitions and results needed for the present work are recalled, and in section 3
the main characterizations are proved.\par 

\section{ Joint spectra of representations of Lie algebras}

\indent In this section the definitions and the main  
properties of the Taylor, the S\l odkowski, the Fredholm, the split and  
the Fredholm split joint spectra are reviewed; for a complete exposition see [9], [5], 
[2], [11], [12] and [4].\par 

\indent From now on $X$ denotes a complex Banach space, $L(X)$ the algebra of all
operators defined in $X$, $K(X)$ the ideal of all compact operators, 
and $L(X,Y)$ the algebra of all bounded linear maps
from $X$ to $Y$, where $Y$ is another Banach space. If $X$
and $Y$ are two Banach spaces and $T\in L(X,Y)$, then the range and the null space of $T$ 
are denoted by $R(T)$ and $N(T)$ respectively.  
In addition, if $L$ is a complex solvable finite dimensional 
Lie algebra, then $\rho\colon L\to  L(X)$ denotes a representation of $L$ in $X$.
Now well, if $X$, $L$ and $\rho$ are as above, it is possible to consider the 
Koszul complex of the representation $\rho$, i.e. 
$(X\otimes\wedge L, d(\rho))$, where $\wedge L$ denotes the 
exterior algebra of $L$ and  
$d_p (\rho)\colon X\otimes\wedge^p L\rightarrow X\otimes
\wedge^{p-1} L$ is the map defined by

\begin{align*} 
d_p (\rho)&( x\otimes\langle l_1\wedge\dots\wedge l_p\rangle) = \sum_{k=1}^p (-1)^{k+1}\rho (l_k)(x)\otimes\langle l_1
\wedge\ldots\wedge\hat{l_k}\wedge\ldots\wedge l_p\rangle\\
    + & \sum_{1\le i< j\le p} (-1)^{i+j-1}x\otimes\langle [l_i, l_j]\wedge l_1\wedge\ldots\wedge\hat{l_i}\wedge\ldots\wedge
\hat{l_j}\wedge\ldots\wedge  l_p\rangle , 
\end{align*}
where $\hat{ }$ means deletion. If $\dim L=n$, then 
for $p$ such that $p\le 0$ or $p\ge {n+1}$, $d_p (\rho) =0$.\par

\indent In addition, if $f$ is a character of $L$, i.e. $f\in L^*$ and 
$f(L^2) = 0$, where $L^2 = \{ [x,y]: x, y \in L\}$ and $[\cdot ,\cdot ]$ denotes
the Lie bracket of $L$, then it is possible to 
consider the representation of $L$ in $X$ given by $\rho-f\equiv \rho-f\cdot I$, 
where $I$ is the identity map of $X$. Now well, if $H_*(X\otimes
\wedge L , d(\rho-f))$ denotes the homology of the Koszul complex of 
the representation $\rho-f$, then it is possible to introduce the sets
$$
\sigma_p(\rho)= \{ f\in L^*: f(L^2)=0, \hbox{  }H_p(X\otimes\wedge L,
d(\rho-f))\ne 0\},
$$
and
$$
\sigma_{p , e}(\rho)=\{f\in L^*: f(L^2)=0, \hbox{  }\dim \hbox{  }H_p(X\otimes
\wedge L,d(\rho-f))=\infty\}.
$$
\indent Next follow the definitions of the Taylor, the S\l odkowski, and the Fredholm joint 
spectra; see [14], [13], [8], [10], [9], [5], [2], [11], [12] and [4].\par

\begin{df} Let $X$ be a complex Banach
space, $L$ a complex solvable finite dimensional Lie algebra,
and $\rho\colon L\to  L(X)$ a representation of $L$
in $X$. Then, 
the Taylor joint spectrum of $\rho$
is the set
$$\sigma(\rho)= \cup_{p=0}^n\sigma_p(\rho)=\{f\in L^*: f(L^2) =0, 
\hbox{  }H_*(X\otimes\wedge L,d(\rho-f))\neq 0\}.
$$

\indent In addition, the $k$-th $\delta$-S\l odkowski joint spectrum of $\rho$ is the 
set
$$
\sigma_{\delta ,k}(\rho)= \cup_{p=0}^k\sigma_p(\rho) ,
$$
and the $k$-th $\pi$-S\l odkowski joint spectrum of $\rho$ is the set
$$
\sigma_{\pi ,k}(\rho)= \cup_{p=n-k}^n\sigma_p(\rho)\cup \{f
\in L^*: f(L^2)=0, \hbox{  }R(d_{n-k}(\rho-f)) \hbox{ is not closed}\},
$$
for $k=0,\ldots , n=\dim L$. \par
\indent Observe that   $\sigma_{\delta ,n}(\rho)= \sigma_{\pi ,n}
(\rho)= \sigma(\rho)$.\par 

\indent On the other hand, the Fredholm or essential Taylor joint spectrum of $\rho$ is 
the set    
$$
\sigma_e(\rho)=\cup_{p=0}^n \sigma_{p , e}(\rho).
$$
\indent In addition, the $k$-th Fredholm or essential $\delta$-S\l odkowski joint spectra of $\rho$ 
is the set
$$
\sigma_{\delta , k ,e}(\rho)=\cup_{p=0}^k \sigma_{p ,e}(\rho),   
$$
and the $k$-th Fredholm or essential $\pi$-S\l odkowski joint spectrum of $\rho$ is
the set
$$
\sigma_{\pi ,k ,e}(\rho)=\cup_{p=n-k}^n \sigma_{p ,e}(\rho)
\cup \{f\in L^*: f(L^2)=0,\hbox{  } 
R(d_{n-k}(\rho))\hbox{  is not closed}\},
$$
for $k=0,\ldots , n$.\par
\indent Observe that $\sigma_e(\rho)=\sigma_{\delta ,n ,e}(\rho)=
\sigma_{\pi ,n ,e}(\rho)$.\par

\end{df}

\indent  In order to state the definition of the split and Fredholm split joint spectra,
some preliminary facts are needed.\par
 
\indent A finite complex of Banach spaces and bounded linear operators 
$(X,d)$ is a sequence 
$$
0\to X_n\xrightarrow{d_n}X_{n-1}\to\ldots\to X_1\xrightarrow{d_1}X_0\to 0,
$$
where $n\in \Bbb N$, $X_p$ are Banach spaces, and the maps 
$d_p \in L(X_p ,X_{p-1})$ are such that $d_{p-1}
\circ d_p=0$,  for $p=1,\ldots , n$. \par
\indent Now well, given a fixed integer $p$, $0\le p\le n$, 
the complex $(X,d)$ is said split (resp.  Fredholm split) in degree $p$, if there are continous 
linear operators
$$
X_{p+1}\xleftarrow{h_p}X_p\xleftarrow{h_{p-1}}X_{p-1},
$$
such that $d_{p+1}h_p+h_{p-1}d_p=I_p$ (resp. $d_{p+1}h_p+h_{p-1}d_p=I_p -k_p$, for
some $k_p\in K(X_p)$), where $I_p$ denotes the identity map of $X_p$,
$p=0, \ldots ,n$; see [7; 2].\par

\indent In addition, if $L$, $X$ and $\rho$ are as above, and if $p$ 
is such that $0\le p\le n$, then it is possible to introduce the sets
$$
sp_p (\rho )= \{ f\in L^*: f(L^2)=0,\hbox{  } (X\otimes\wedge L,d(\rho-f))
\hbox{  is not split in degree p} \}
$$
and
\begin{align*}
sp_{p,e} (\rho )= &\{ f\in L^*: f(L^2)=0, \hbox{  }(X\otimes\wedge L,d(\rho-f))
\hbox{  is not Fredholm split } \\
                          &\hbox{  in degree p} \}.
\end{align*}
\indent Next follow the definitions of the split and the Fredholm split joint spectra; see 
[7], [12] and [4].\par

\begin{df}Let $X$ be a complex Banach
space, $L$ a complex solvable finite dimensional Lie algebra,
and $\rho\colon L\to  L(X)$ a representation of $L$
in $X$. Then, 
the split spectrum of $\rho$ is the set
$$
sp(\rho)=\cup_{p=0}^n sp_p(\rho).
$$
\indent In addition, the $k$-th $\delta$-split joint spectrum of $\rho$ is the set    
$$
sp_{\delta ,k}(\rho)= \cup_{p=0}^k sp_p(\rho) ,
$$
and the $k$-th $\pi$-split joint spectrum of $\rho$ is the set
$$
sp_{\pi ,k}(\rho)= \cup_{p=n-k}^n sp_p(\rho),
$$
for $k=0,\ldots , n=\dim L$.\par
\indent  Observe that  $sp_{\delta ,n}(\rho)= sp_{\pi ,n}(\rho)= 
sp(\rho)$.\par
\indent On the other hand, the Fredholm or essential split spectrum of $\rho$ is the set
$$
sp_e(\rho)=\cup_{p=0}^n sp_{p,e}(\rho).
$$
\indent In addition, the $k$-th Fredholm or essential $\delta$-split joint spectrum of $\rho$ 
is the set    
$$
sp_{\delta ,k,e}(\rho)= \cup_{p=0}^k sp_{p,e}(\rho) ,
$$
and the $k$-th Fredholm or essential $\pi$-split joint spectrum of $\rho$ is the set
$$
sp_{\pi ,k,e}(\rho)= \cup_{p=n-k}^n sp_{p,e}(\rho),
$$
for $k=0,\ldots , n$.\par
\indent  Observe that   $sp_{\delta ,n,e}(\rho)= 
sp_{\pi ,n,e}(\rho)= sp_e(\rho)$.\par
\end{df}

\indent It is clear that $\sigma_{\delta ,k}(\rho)\subseteq 
sp_{\delta ,k}(\rho)$, $\sigma_{\pi ,k}(\rho)\subseteq 
sp_{\pi ,k}(\rho)$, $\sigma(\rho)\subseteq sp(\rho)$, and 
$\sigma_{\delta ,k, e}(\rho)\subseteq 
sp_{\delta ,k, e}(\rho)$, $\sigma_{\pi ,k,e}(\rho)\subseteq 
sp_{\pi ,k,e}(\rho)$, $\sigma_e(\rho)\subseteq sp_e(\rho)$. 
Moreover, if $X$ is a Hilbert space, then the above inclusions are equalities.\par 

\indent All the above considered joint spectra are defined 
for representations of complex solvable finite dimensional Lie algebras
in complex Banach spaces and have the main spectral 
properties. That is, they are compact nonempty subsets of 
characters of the Lie algebra $L$ and they have 
the projection property for ideals, i.e.
if $L$ is such a Lie algebra, $I$ is an ideal of $L$, $\sigma_*$
is one of above joint spectra, and $\pi\colon L^*\to I^*$ denotes the 
restriction map from $L^*$ to $I^*$, then  
$$
\pi (\sigma_*(L))=\sigma_*(I).
$$
For a complete exposition see the works [9], [5], [2], [11], [12] and [4].\par 
\indent In the following section representations of nilpotent Lie algebras by compact operators
in finite or infinite dimensional Banach spaces will be considered and all the above mentioned joint spectra of such a representation
will be computed. However,
to this end, first it is necessary to 
recall a result from [4].\par

\indent If $L$, $X$, and $\rho\colon L
\to L(X)$ are as above, then it is possible to consider the 
representation 
$$
L_{\rho}\colon L\to  L(L(X)),\hskip1cm l\to L_{\rho(l)},
$$
where $L_{\rho(l)}$ denotes the left multiplication
operator associated to $\rho(l)$, $l\in L$. 
In addition, since $L_{\rho(l)}( K(X))\subseteq 
 K(X)$, it is possible to consider the representation
$$
\tilde{L}_{\rho}\colon L\to  L(C(X)),
$$   
where $C(X)=L(X)/ K(X)$,
and $\tilde{L}_{\rho}(l)$ is the quotient operator defined in 
$C(X)$ associated to $L_{\rho(l)}$, $l\in L$.\par 

\indent Similarly, if $L^{op}$ is the Lie algebra which as vector space
coincides with $L$ but whose Lie bracket is the opposite one, 
then it is possible to consider the representation 
$$
R_{\rho}\colon L^{op}\to L( L(X)),\hskip1cm l\to 
R_{\rho(l)},
$$
where $R_{\rho(l)}$ denotes the right multiplication
operator associated to $\rho(l)$, $l\in L^{op}$. 
Furthermore, since $R_{\rho(l)}(K(X))\subseteq  K(X)$, it is possible to consider the 
representation
$$
\tilde{R}_{\rho}\colon L^{op}\to  L( C(X)),
$$   
where $\tilde{R}_{\rho}(l)$ is the quotient 
operator defined in 
$C(X)$ associated to $R_{\rho(l)}$, $l\in L^{op}$.\par 

\indent Now well, if $L$ is a nilpotent Lie algebra, then according to [4; 8]
and [12; 0.5.8]:
\noindent i- $sp_{\delta ,k,e}(\rho)=\sigma_{\delta ,k}
(\tilde{L}_{\rho})$,\par

\noindent ii- $sp_{\pi ,k,e}(\rho)=\sigma_{\delta , k}
(\tilde{R}_{\rho})$,\par

\noindent iii- $sp_e(\rho)=\sigma
(\tilde{L}_{\rho})=\sigma(\tilde {R}_{\rho})$.\par

\section{The main results}
\indent In this section, given a
representation of a complex nilpotent finite dimensional
Lie algebra by compact operators in a finite or infinite 
complex Banach space, the Taylor the S\l odkowski, the Fredholm,
the split, and the Fredholm split joint spectra of such a representation are characterized. 
Since the example
which follows Theorem 5 in [3] shows that for a solvable non-nilpotent
Lie algebra of compact operators the characterization fails,
only nilpotent Lie algebras are considered.
In first place the infinite dimensional case is studied. \par

\indent In the following theorem the Taylor and the 
S\l odkowski
joint spectra are described.\par 
 
\begin{thm} Let $X$ be an infinite dimensional complex Banach space,
$L$ a complex nilpotent finite dimensional
Lie algebra, and $\rho\colon L\to L(X)$ a representation
of $L$ in $X$ such that $\rho (l)\in K(X)$ for each $l\in L$. 
Then, the sets $\sigma(\rho)$, 
$\sigma_{\delta ,k}(\rho)\cup \{ 0\}$,
and $\sigma_{\pi ,k}(\rho)\cup \{ 0\}$ coincide with the set

\begin{align*}
\{0\}\cup\{& f\in L^*: f(L^2)=0, \hbox{ such that
there is  }x\in X,\hbox{ }x\neq 0, \hbox{  with the} \\
&\hbox{property:  } \rho(l)(x)=f(l)x,\hbox{  } \forall\hbox{  } l\in L\},
\end{align*}

for $k=0, \ldots , n=\dim L$,

 \end{thm}

\begin{proof}
\indent First of all recall that, according to [1; 3.8], [1; 3.9] and [9; 2.6], 
\begin{align*}
\sigma(\rho) =\{ 0 \}\cup&\{f\in L^*: f( L^2)=0, 
\hbox{  such that there is  }x\in X,\hbox{  }x\neq 0,\\ 
& \hbox{  with the property:  } 
\rho(l)(x)=f(l)x,\hbox{  } \forall\hbox{  } l\in L \}.\end{align*}

\indent Now, since $L$ is a nilpotent Lie algebra, according to the 
above equality and Definition 1, 
$\sigma (\rho)=\{ 0 \}\cup \sigma_{\pi ,0}(\rho)$. However, since
$\sigma_{\pi ,0}(\rho)\cup \{ 0 \}\subseteq\sigma_{\pi ,k}(\rho)\cup \{ 0 \}
\subseteq \sigma (\rho)$, $0\le k\le n$, 
$$
\sigma(\rho) =\sigma_{\pi ,k}\cup\{ 0\}=\sigma_{ \pi ,0}(\rho)\cup\{ 0\}.
$$

\indent On the other hand, if 
the adjoint representation of $\rho$ is considered, i.e. 
the representation $\rho^*\colon L^{op}
\to L(X{'})$, $\rho^*(l)=(\rho (l))^*$, where $X{'}$ denotes the dual space 
of $X$, then since $L^{op}$ is a nilpotent Lie algebra, according to [12; 0.5.8] and [12; 2.11.4],  
$$
\{0\}\cup\sigma_{\delta ,k}(\rho)=\{0\}\cup
\sigma_{\pi ,k}(\rho^*)=\{0\}\cup \sigma_{\pi ,0}(\rho^*)=
 \sigma(\rho^*)=\sigma (\rho).
$$
\end{proof}
\indent In the following theorem the Fredholm joint spectra
are computed.\par

\begin{thm} Let $X$ be an infinite dimensional
complex Banach space, $L$ a complex nilpotent finite dimensional
Lie algebra, and $\rho\colon L\to L(X)$ a representation
of $L$ in $X$ such that $\rho(l)\in K(X)$ for each $l\in L$. 
Then  
$$
\sigma_e(\rho)=\sigma_{\delta ,k ,e}(\rho)=\sigma_{\pi ,k ,e}(\rho)=
\{0\},
$$
for $k=0, \ldots , n=\dim L$.\par
\end{thm}
\begin{proof}
\indent The proof is based on an induction argument on the dimension of the algebra.\par

\indent If $\dim L=1$, then consider $l\in L$ such that $<l>=L$. Now,
since $\rho (l)$ is a compact operator,  
$\sigma_e(\rho(l))=\{0 \}$.
However, since $\dim L=1$, according to Definition 1, $\sigma_e (\rho)=\{ 0\}$.
Now, since $\sigma_{\delta ,k, e}(\rho)$ and $\sigma_{\pi ,k ,e}(\rho)$
are nonempty subsets of $\sigma_e(\rho)$, $0\le k\le 1=\dim L$, 
$$
\sigma_e(\rho)=\sigma_{\delta ,k ,e}(\rho)= 
\sigma_{\pi ,k ,e}(\rho)=\{0\},
$$
$0\le k\le 1=\dim L$.\par

\indent Now suppose that for every nilpotent Lie algebra of
dimension less than $n$ and for every representation of the algebra
by compact operators in an infinite dimensional Banach space $X$  
$$
\sigma_e(\rho)=\sigma_{\delta ,k,e}(\rho)= 
\sigma_{\pi ,k ,e}(\rho)=\{0\}.
$$
\indent Now well, if $L$ is a nilpotent Lie algebra of dimension $n$,
then according to [6; 5.1] there is a Jordan-H\" older sequences of ideals $(L_i)_{0\le i
\le n}$ such that \par
\noindent i- $L_0=0$, and $L_n=L$, \par
\noindent ii- $L_i\subseteq L_{i+1}$, $0\le i\le n-1$,\par
\noindent iii- $[L_i, L_j]\subseteq L_{i-1}$, for $i<j$.\par

\indent It is easy to prove that $L^2\subseteq L_{n-2}$.
Then, it is possible to consider the ideals $I_1=L_{n-1}$ and $I_2=
L_{n-2}\oplus <x>$, where $x\in L$ is such that 
$L_{n-1}\oplus<x>=L$.\par

\indent On the other hand, if the representations $\rho_1
=\rho\mid I_1\colon I_1\to L(X)$ and $\rho_2=\rho\mid I_2
\colon I_2\to L(X)$ are considered, then, according to the assumption under
consideration, 
$$
\sigma_e(\rho_1)=\sigma_{\delta ,k ,e}(\rho_1)= 
\sigma_{\pi ,k, e}(\rho_1)=\{0\},
$$
and 
$$
\sigma_e(\rho_2)=\sigma_{\delta ,k ,e}(\rho_2)= 
\sigma_{\pi ,k ,e}(\rho_2)=\{0\}.
$$

However, according to the projection property of the Fredholm joint spectra, 
[4; 3.2] and [4; 3.5],  
$$
\sigma_e(\rho)=\sigma_{\delta ,k ,e}(\rho)= 
\sigma_{\pi ,k ,e}(\rho)=\{0\},
$$
for $k=0,\ldots , n=\dim L$.\par
\end{proof} 

\indent In the following theorem the split and the
Fredholm split joint spectra are computed.\par

\begin{thm} Let $X$ be an infinite dimensional
complex Banach space, $L$ a complex nilpotent finite dimensional
Lie algebra, and $\rho\colon L\to L(X)$ a representation
of $L$ in $X$ such that $\rho(l)\in K(X)$ for each $l\in L$. 
Then, $\sigma (\rho)=sp(\rho)$, $\sigma_{\delta ,k}(\rho)=sp_{\delta ,k}(\rho)$
and $\sigma_{\pi ,k}(\rho)=sp_{\pi ,k}(\rho)$,where $k=0,\ldots , n=\dim L$.\par
\indent In particular, the sets $\sigma (\rho)$, $\sigma_{\delta ,k}(\rho)\cup \{0\}$,
$\sigma_{\pi ,k}(\rho)\cup \{0\}$, $sp(\rho)$, $sp_{\delta ,k}(\rho)\cup \{0\}$
and $sp_{\pi ,k}(\rho)\cup \{0\}$ coincide with the set

\begin{align*}
\{0\}\cup\{& f\in L^*: f(L^2)=0, \hbox{  such that there
is  }x\in X,\hbox{  }x\neq 0, \\
&\hbox{with the property  }\rho(l)(x)=f(l)x,\hbox{  } \forall\hbox{  } l\in L\}.\end{align*}

\indent In addition, 
$$
sp_{e}(\rho)=sp_{\delta ,k ,e}(\rho)=sp_{\pi ,k,e}(\rho)=\{0\},
$$
where $k=0,\ldots , n=\dim L$.\par
\indent In particular, all the Fredholm and Fredholm split joint spectra 
coincide with the set $\{0\}$.
\end{thm}
\begin{proof}

\indent First of all, since $\rho (L)\subseteq K(X)$, $\tilde{L}_{\rho}=0$
and $\tilde{R}_{\rho}=0$. Thus, since $L$ is a nilpotent Lie algebra,
according to [4; 8] and [12; 0.5.8], \par
\noindent i- $sp_{\delta ,k,e}(\rho)=\sigma_{\delta ,k}(\tilde{L}_{\rho})=\{0\}$,\par
\noindent ii- $sp_{\pi ,k,e}(\rho)=\sigma_{\delta, k}(\tilde{R}_{\rho})=\{0\}$,\par
\noindent iii- $sp_e(\rho)=\sigma(\tilde{L}_{\rho})=\sigma(\tilde{R}_{\rho})=\{0\}$.\par
\indent Now well, in order to prove the first assertion of the Theorem, observe that,
according to [7; 2.7], \par
\noindent i- $sp(\rho)=\sigma (\rho)\cup {\mathcal C}$,
$sp_e(\rho)=\sigma_e (\rho)\cup \tilde{{\mathcal C}}$,

\noindent ii- $sp_{\delta ,k}(\rho)=\sigma_{\delta ,k} (\rho)
\cup {\mathcal A}_k$,
$sp_{\delta ,k ,e}(\rho)=\sigma_{\delta ,k,e} (\rho)
\cup\tilde{ {\mathcal A}}_k$,\par
\noindent iii- $sp_{\pi ,k}(\rho)=\sigma_{\pi ,k} (\rho)\cup {\mathcal B}_k$,
$sp_{\pi ,k,e}(\rho)=\sigma_{\pi ,k,e} (\rho)\cup\tilde{ {\mathcal B}}_k$,
where $k=0,\ldots , n$ and\par
 \noindent iv- ${\mathcal A}_k=\{f\in L^*: f(L^2)=0,\hbox{ }  f\notin
\sigma_{\delta ,k}(\rho), \hbox{ and there is } p,\hbox{ } p=1,\ldots
, k+1,\hbox{ such that } N(d_p(\rho-f))
\hbox{ is not complemented in } X\otimes \wedge^p L\}$,
\par
\noindent v- $\tilde{{\mathcal A}}_k=\{f\in L^*: f(L^2)=0,  \hbox{ }f\notin
\sigma_{\delta ,k ,e}(\rho), \hbox{ and there is } p,\hbox{ }
p=1,\ldots ,k+1, \hbox{ such that }N(d_p(\rho-f))
\hbox{ is not complemented in }X\otimes \wedge^p L\}$,
\par
\noindent vi- ${\mathcal B}_k=\{f\in L^*: f(L^2)=0, \hbox{ }f\notin\sigma_{\pi ,k}
(\rho), \hbox { and there is } p,\hbox{ } p=n,\ldots , n-k,
\hbox{ such that }R(d_p(\rho-f))\hbox{ is not complemented in } X\otimes \wedge^{p-1}L \}$, \par
\noindent vii- $\tilde{{\mathcal B}}_k=\{f\in L^*: f(L^2)=0, \hbox{ }f\notin
\sigma_{\pi ,k ,e}(\rho), \hbox { and there is } p, \hbox{ } p=n,\ldots ,n-k,
\hbox{ such that } R(d_p(\rho-f))\hbox{ is not complemented in }X\otimes \wedge^{p-1} L 
\}$,\par
\noindent viii- ${\mathcal C}={\mathcal A}_n={\mathcal B}_n$,
$\tilde{{\mathcal C}}=\tilde{{\mathcal A}}_n=\tilde{{\mathcal B}}_n$.\par

\indent In addition,   
${\mathcal A}_k\subseteq\tilde{{\mathcal A}}_k$, and 
${\mathcal B}_k\subseteq\tilde{{\mathcal B}}_k$, for $k=0,\ldots , n$.\par

\indent However, since $\sigma_{\delta ,k ,e}(\rho)\cap
\tilde{\mathcal A}_k=\emptyset$ and $\sigma_{\pi ,k ,e}(\rho)\cap
\tilde{\mathcal B}_k=\emptyset$, according to Theorem 2 and to the second
part of the Theorem, which have already been proved,  
$\tilde{\mathcal A}_k=\emptyset$ and $\tilde{\mathcal B}_k=\emptyset$.
In particular, ${\mathcal A}_k=\emptyset$ and ${\mathcal B}_k=\emptyset$,
which proves the first assertion of the Theorem.\par  
\end{proof} 
 
\indent Next representations of nilpotent Lie algebras in finite
dimensional Banach spaces are considered. Observe that in this case the
spectra to be studied are the Taylor and the S\l odkowski
joint spectra.\par
\begin{thm} Let $X$ be a complex finite dimensional Banach space,
$L$ a complex nilpotent finite dimensional Lie algebra, and $\rho\colon L\to L(X)$ 
a representation of $L$ in $X$. Then

\begin{align*}
\sigma(\rho)=\sigma_{\delta ,k}(\rho)=\sigma_{\pi, k}(\rho)= &\{f\in L^*: f(L^2)=0, 
\hbox{ such that there is } x\in X,\\
                    & x\neq 0, \hbox { with the property: }\rho (l)x=f(l)x,\hbox{  }\forall\hbox{  } l\in L\}.
\end{align*}
 \end{thm}
\begin{proof}
\indent Since $L$ is a nilpotent Lie algebra, according to [3; 4] and [9; 2.6],

\begin{align*}
\sigma(\rho)= &\{f\in L^*: f(L^2)=0,\hbox{ such that there is } x\in X,\hbox{ } x\neq 0, 
\hbox { with the}\\
                    &\hbox{property: }\rho (l)x=f(l)x,\forall\hbox{  } l\in L\}.\end{align*}
\indent Therefore, since $L$ is a nilpotent Lie algebra, $\sigma(\rho)=
\sigma_{\pi , 0}(\rho)$. Now well, in order to finish the proof, it is possible to use an argument
similar to the one developed in Theorem 1.\par
\end{proof}

\bibliographystyle{amsplain}

\vskip.5cm

Enrico Boasso\par
E-mail address: enrico\_odisseo@yahoo.it

\end{document}